\batchmode
\makeatletter
\def\input@path{{"/Users/russw/Documents/Research/mypapers/Shellings from relative shellings/"}}
\makeatother
\documentclass[12pt,oneside,english]{amsart}
\usepackage[T1]{fontenc}
\usepackage[latin9]{inputenc}
\usepackage{geometry}
\usepackage{babel}
\usepackage{url}
\usepackage{enumitem}
\usepackage{amstext}
\usepackage{amsthm}
\usepackage{amssymb}
\usepackage{graphicx}
\usepackage[pdftex,unicode=true,pdfusetitle,
 bookmarks=true,bookmarksnumbered=false,bookmarksopen=false,
 breaklinks=false,pdfborder={0 0 1},backref=false,colorlinks=false]
 {hyperref}

\makeatletter
\numberwithin{equation}{section}
\numberwithin{figure}{section}
\theoremstyle{plain}
\newtheorem{thm}{\protect\theoremname}[section]
\theoremstyle{plain}
\newtheorem{lem}[thm]{\protect\lemmaname}
\theoremstyle{remark}
\newtheorem{rem}[thm]{\protect\remarkname}
\theoremstyle{definition}
\newtheorem{example}[thm]{\protect\examplename}
\theoremstyle{plain}
\newtheorem{prop}[thm]{\protect\propositionname}
\theoremstyle{plain}
\newtheorem{cor}[thm]{\protect\corollaryname}

\usepackage{lmodern}

\makeatother

\providecommand{\corollaryname}{Corollary}
\providecommand{\examplename}{Example}
\providecommand{\lemmaname}{Lemma}
\providecommand{\propositionname}{Proposition}
\providecommand{\remarkname}{Remark}
\providecommand{\theoremname}{Theorem}

\begin{document}
\global\long\def\cc{\mathbb{C}}%

\global\long\def\link{\operatorname{link}}%

\global\long\def\Tau{\mathrm{T}}%

\global\long\def\Beta{\mathrm{B}}%

\global\long\def\TSAT{\mathsf{3SAT}}%

\global\long\def\NPtime{\mathsf{NP}}%

\global\long\def\coNPtime{\mathsf{coNP}}%

\global\long\def\Ptime{\mathsf{P}}%

\global\long\def\SHELL{\mathsf{SHELLABILITY}}%

\title[Shellings from relative shellings]{Shellings from relative shellings,\\
with an application to NP-completeness}
\author{Andrés Santamaría-Galvis and Russ Woodroofe}
\thanks{This work is supported in part by the Slovenian Research Agency (research
program P1-0285 and research projects J1-9108, N1-0160, and J1-2451). }
\address{Univerza na Primorskem, Glagoljaška 8, 6000 Koper, Slovenia}
\email{andres.santamaria@famnit.upr.si}
\urladdr{\url{https://sites.google.com/view/adsantamaria/}}
\address{Univerza na Primorskem, Glagoljaška 8, 6000 Koper, Slovenia}
\email{russ.woodroofe@famnit.upr.si}
\urladdr{\url{https://osebje.famnit.upr.si/~russ.woodroofe/}}
\begin{abstract}
Shellings of simplicial complexes have long been a useful tool in
topological and algebraic combinatorics. Shellings of a complex expose
a large amount of information in a helpful way, but are not easy to
construct, often requiring deep information about the structure of
the complex. It is natural to ask whether shellings may be efficiently
found computationally. In a recent paper, Goaoc, Paták, Patáková,
Tancer and Wagner gave a negative answer to this question (assuming
$\Ptime\neq\NPtime$), showing that the problem of deciding whether
a simplicial complex is shellable is $\NPtime$-complete.

In this paper, we give simplified constructions of various gadgets
used in the $\NPtime$-completeness proof of these authors. Using
these gadgets combined with relative shellability and other ideas,
we also exhibit a simpler proof of the $\NPtime$-completeness of
the shellability decision problem. Our method systematically uses
relative shellings to build up large shellable complexes with desired
properties.
\end{abstract}

\maketitle

\section{\label{sec:Introduction}Introduction}

A \emph{shelling }is a certain way of building up (or equivalently,
tearing down) a simplicial complex, facet by facet. A precise definition
may be found in Section~\ref{sec:Background}. Shellings have found
considerable application by combinatorialists and combinatorial algebraists.
In the 1970s, work of Hochster, Reisner, and Stanley showed how to
use shellings \cite{Hochster:1972,Reisner:1976,Stanley:1975a} to
prove that certain rings are Cohen-Macaulay; this is still a useful
tool \cite{Herzog/Hibi:2011,Miller/Sturmfels:2005,Villarreal:2015}.
The existence of a shelling makes computing homotopy type easy, and
the topology (up to homeomorphism) tractable in many cases \cite{Bjorner:1995,Chari:1997,Wachs:2007}.
In certain cases, the existence of a shelling is equivalent to the
existence of other interesting structure: for example, the order complex
of a finite group is shellable if and only if the group in question
is solvable \cite{Shareshian:2001}.

In a pair of papers \cite{Danaraj/Klee:1978b,Danaraj/Klee:1978a}
from the 1970s, Danaraj and Klee consider the decision problem $\SHELL$,
that is, the problem of determining whether a simplicial complex is
shellable. They showed that $\SHELL$ is in $\Ptime$ when restricted
to $2$-dimensional pseudomanifolds, and suggested that the general
problem might be $\NPtime$-complete. See also \cite{Kaibel/Pfetsch:2003}.
This problem sat open for 40 years, until Goaoc, Paták, Patáková,
Tancer and Wagner, in a significant recent advance \cite{Goaoc/Patak/Patakova/Tancer/Wagner:2018,Goaoc/Patak/Patakova/Tancer/Wagner:2019},
verified the problem to be $\NPtime$-complete:
\begin{thm}[{Goaoc, Paták, Patáková, Tancer, and Wagner \cite[Theorem 1]{Goaoc/Patak/Patakova/Tancer/Wagner:2019}}]
$\qquad$\linebreak{}
\label{thm:ShellNPcomplete}$\SHELL$ is $\NPtime$-complete, even
when restricted to $2$-dimensional simplicial complexes.
\end{thm}

The proof is by polynomial reduction from $\TSAT$. As is typical
in such a reduction, the construction in \cite{Goaoc/Patak/Patakova/Tancer/Wagner:2019}
proceeds by building ``choice gadgets'' (corresponding to variables
in a $\TSAT$ instance), ``constraint gadgets'' (corresponding to
clauses), and some other needed gadgets for consistency. As an essential
building block, these authors require simplicial complexes that are
shellable, but where the shelling order is `rigid' in a certain precise
sense.

One such building block is a shellable complex with a single free
face $\tau$; in such a complex, every shelling order must end with
the facet containing $\tau$. The authors of \cite{Goaoc/Patak/Patakova/Tancer/Wagner:2019}
use a construction based on Bing's house with 2 rooms. This construction
is originally due to Malgouyres and Francés in \cite{Malgouyres/Frances:2008}.
As observed in \cite[Remark 9]{Goaoc/Patak/Patakova/Tancer/Wagner:2019},
there is also a somewhat more concrete construction due to Hachimori
\cite{Hachimori:2000}. We mention that a construction with similar
properties was earlier considered by Simon \cite[Appendix F]{Simon:1994},
and that Hachimori's construction was extended to arbitrary dimensions
by Adiprasito, Benedetti and Lutz \cite[Theorem 2.3]{Adiprasito/Benedetti/Lutz:2017}.

Another building block in \cite{Goaoc/Patak/Patakova/Tancer/Wagner:2019}
is a shellable complex with 3 nonadjacent free faces, and satisfying
some other conditions needed for gluing it to a larger construction.
The authors in \cite{Goaoc/Patak/Patakova/Tancer/Wagner:2019} roughly
sketch a construction based on Bing's house for the latter building
block, but do not provide all details. They instead give a reference
to another paper of Tancer \cite[Section 4]{Tancer:2016}, which provides
more details of the construction. The resulting simplicial complex
is fairly large and complicated.

The first result on this paper will be to give a simple and explicit
construction of shellable complexes with an arbitrary number of nonadjacent
free faces.
\begin{thm}
\label{thm:TurbinesExist}For any positive integer $n$, there is
a simplicial complex $\Tau^{(n)}$ and subcomplex $\Upsilon$ such
that:
\begin{enumerate}
\item $\Tau^{(n)}$ is a shellable and contractible $2$-dimensional complex,
having exactly $n$ free edges and no other free faces. Denote by
$\mathcal{F}$ the set of free edges.
\item \label{enu:TurbinesTree}$\Upsilon\cup\left\langle \mathcal{F}\right\rangle $
is a shellable and contractible $1$-dimensional complex (that is,
a tree), having exactly $n$ leaves; and so that for each leaf vertex,
the unique edge containing it is in $\mathcal{F}$.
\item \label{enu:TurbinesRelShell}For any proper subset $\mathcal{E}\subsetneq\mathcal{F}$
of free edges, the relative complex $(\Tau^{(n)},\Upsilon\cup\left\langle \mathcal{E}\right\rangle )$
is shellable.
\end{enumerate}
\end{thm}

An easy homology calculation gives (in the notation of the theorem)
that $(\Tau^{(n)},\Upsilon\cup\left\langle \mathcal{F}\right\rangle )$
is not shellable, so the word ``proper'' in Theorem~\ref{thm:TurbinesExist}~(\ref{enu:TurbinesRelShell})
cannot be removed. We remark also that Theorem~\ref{thm:TurbinesExist}
essentially restates in terms of relative shellability (and in a more
general context) the conditions required by \cite{Goaoc/Patak/Patakova/Tancer/Wagner:2019}.

Our second result will be an improved construction for and proof of
Theorem~\ref{thm:ShellNPcomplete}. Our construction improves on
that of \cite{Goaoc/Patak/Patakova/Tancer/Wagner:2019} in several
ways. We significantly simplify the needed choice gadgets, eliminating
a consistency gadget needed in the earlier paper, and markedly reducing
the size and complexity. See Remark~\ref{rem:ImprovementsOverPrior}.
Of course, we also use the improved building blocks of Theorem~\ref{thm:TurbinesExist}. 

Our proof of Theorem~\ref{thm:ShellNPcomplete} is also somewhat
differently structured from that of \cite{Goaoc/Patak/Patakova/Tancer/Wagner:2019}.
The authors of this paper give all proofs in terms of collapsibility,
rather than shellability (using a result of Hachimori \cite[Theorem 8]{Hachimori:2008}).
We phrase our proof in the language of shellability and relative shellability,
which we believe some readers may prefer. 

The main innovation that we introduce is the systematic use of relative
shellings to build up large shellable complexes with desired properties.
We use this approach in both the proof of Theorem~\ref{thm:TurbinesExist}
as well as in our new proof of Theorem~\ref{thm:ShellNPcomplete}.
We believe that our techniques may find application to other problems.

We mention that Danaraj and Klee asked a more specific question than
that answered by \cite{Goaoc/Patak/Patakova/Tancer/Wagner:2019}:
is $\SHELL$ $\NPtime$-complete when restricted to $d$-pseudomanifolds
(for some $d>2$)? We don't know the answer to this, but believe that
relative shellability ideas similar to those we use here may be helpful
in addressing the problem. A similar question that might be interesting
to consider: does $\SHELL$ remain $\NPtime$-complete when restricted
to complexes with an embedding in $\mathbb{R}^{3}$? $\mathbb{R}^{4}$?
Other spaces? 

Although $\SHELL$ of a $2$-dimensional complex is $\NPtime$-complete,
the problem restricted to a $2$-dimensional ball or sphere is trivial,
and in this situation a shelling can be computed in linear time \cite{Danaraj/Klee:1978b}.
Is $\SHELL$ for a $3$-dimensional ball or sphere in $\Ptime$? We
remark that one quite general construction of nonshellable $3$-balls
uses nontrivial knots as an essential ingredient (see e.g. \cite[Section XIV.6]{Bing:1983}).
As the $\mathsf{KNOTTEDNESS}$ problem has been shown to be in $\coNPtime$
\cite{Kuperberg:2014,Lackenby:2016UNP}, it seems plausible that the
restriction of $\SHELL$ to $3$-balls is also is in $\coNPtime$.

The paper is organized as follows. In Section~\ref{sec:Background}
we give general background on shellability and relative shellability.
We believe that some of the lemmas on relative shellability in this
section may be of broader use. In Section~\ref{sec:Turbines} we
construct the complexes as in Theorem~\ref{thm:TurbinesExist}. In
Section~\ref{sec:NPcomplete} we use these complexes and other ideas
to give our new proof of Theorem~\ref{thm:ShellNPcomplete}.

\section*{Acknowledgements}

We thank an anonymous referee for careful and thoughtful comments,
which helped us improve readability. We also thank Abhishek Rathod
for his helpful feedback on a draft of the paper. 

\section{\label{sec:Background}Background}

As usual, a \emph{simplicial complex} $\Delta$ is a family of sets
(called \emph{faces}) that is closed under inclusion. The simplicial
complex \emph{generated} by a family of sets $\mathcal{E}$ consists
of all subsets of sets in $\mathcal{E}$, and is denoted $\left\langle \mathcal{E}\right\rangle $.
The \emph{$f$-vector} of a simplicial complex $\Delta$ is $\left(f_{0},f_{1},\dots,f_{c}\right)$
where $f_{i}$ is the number of faces of $\Delta$ with $i$ vertices,
and the \emph{$h$-vector} of $\Delta$ is determined by $\sum h_{i}x^{c-i}=\sum f_{i}(x-1)^{c-i}$.
We denote as $f(\Delta)$ and $h(\Delta)$, respectively. Note that
some authors index the $f$-vector by dimension, rather than cardinality.

A \emph{free face} in $\Delta$ is a face $\tau\neq\emptyset$ which
is properly contained in exactly one facet (maximal face). It is easy
to show that if $\tau$ is a free face, then $\Delta$ deformation
retracts to the complex $\Delta\setminus\tau$ obtained by removing
from $\Delta$ all faces containing $\tau$.

A simplicial complex $\Delta$ is \emph{shellable} if there is an
ordering $\sigma_{1},\sigma_{2},\dots,\sigma_{m}$ (a \emph{shelling})\emph{
}of the facets such that each $\sigma_{j}$ with $j>1$ intersects
with the complex $\Delta_{j-1}$ generated by $\sigma_{1},\dots,\sigma_{j-1}$
in a pure $\left(\dim\sigma_{j}\right)-1$ complex. A useful equivalent
condition is as follows: 
\begin{equation}
\forall j\,\forall i<j\,\exists k<j\,\text{ such that }\sigma_{i}\cap\sigma_{j}\subseteq\sigma_{k}\cap\sigma_{j}=\sigma_{j}\setminus\left\{ x\right\} \text{, some }x\in\sigma_{j}.\label{eq:ShellCond}
\end{equation}
 Each facet $\sigma_{j}$ in a shelling contains a minimal ``new''
face, given by \linebreak{}
$\left\{ x\in\sigma_{j}\,:\,\sigma_{j}\setminus\left\{ x\right\} \subseteq\sigma_{k},\text{ some }k<j\right\} $.
If the minimal new face is $\sigma_{j}$ itself, then we say $\sigma_{j}$
is a \emph{homology facet}, otherwise, $\sigma_{j}$ contains a face
that is free in some earlier $\Delta_{k}$. 

A \emph{relative simplicial complex} is a pair $(\Delta,\Gamma)$
of simplicial complexes, where $\Gamma$ is a subcomplex of $\Delta$.
The\emph{ faces} of $\left(\Delta,\Gamma\right)$ are the faces of
$\Delta$ that are not faces of $\Gamma$. It may be helpful to recall
from algebraic topology that there is a homology theory for relative
complexes, and that $\tilde{H}_{i}(\Delta,\Gamma)$ is isomorphic
to the homology $\tilde{H}_{i}(\Delta/\Gamma)$ of the quotient of
$\Delta$ by $\Gamma$.

A relative simplicial complex $\Psi=\left(\Delta,\Gamma\right)$ is
\emph{(relatively) shellable} if there is an ordering $\sigma_{1},\sigma_{2},\dots,\sigma_{m}$
(a \emph{shelling}) of the facets of $\Psi$ such that each $\sigma_{j}$
contains a unique minimal ``new'' face \cite[Chapter III.7]{Stanley:1996}.
That is, the set of subsets of each $\sigma_{j}$ that are not in
a preceding $\sigma_{i}$ or in $\Gamma$ has a unique minimal element
under inclusion. Previous work on relative shellability includes \cite{Adiprasito/Benedetti:2017,Adiprasito/Sanyal:2016,Codenotti/Katthan/Sanyal:2019,Stanley:1996b,White:2018unp}. 

We will sometimes refer to shellability of a simplicial complex as
\emph{absolute shellability}, in order to contrast with shellability
of a relative complex. We see that absolute shellability is the special
case of relative shellability where $\Gamma=\emptyset$. 

We find it more convenient to work with a formulation of relative
shellability that is closer to the standard definition for the absolute
case. The proof is essentially the same as in the absolute case \cite[Section 2]{Bjorner/Wachs:1996}.
\begin{lem}
\label{lem:RelShellRight}Let $\Gamma\neq\emptyset$. An ordering
$\sigma_{1},\dots,\sigma_{m}$ of the facets of a relative simplicial
complex $\Psi=\left(\Delta,\Gamma\right)$ is a shelling if and only
if each $\sigma_{j}$ intersects in a pure $\left(\dim\sigma_{j}\right)-1$
complex with the complex $\Delta_{j-1}$ generated by $\sigma_{1},\dots,\sigma_{j-1}$
together with $\Gamma$.
\end{lem}

\begin{proof}
If $\Psi$ satisfies the condition, then any face of $\sigma_{j}$
either contains $\gamma=\left\{ x\,:\,\sigma_{j}\setminus x\in\Delta_{j-1}\right\} $,
or else is a face of $\Delta_{j-1}$. Thus, $\gamma$ is the required
minimal new face. Conversely, if $\Psi$ is shellable with minimal
new face $\gamma$, then the intersection with $\Delta_{j-1}$ is
generated by the faces $\sigma_{j}\setminus v$ over $v\in\gamma$.
It follows that the intersection is pure of codimension one, as required.
\end{proof}
Now the immediate analogue of (\ref{eq:ShellCond}) is as follows.
Let $\Delta_{j-1}$ be as in the statement of Lemma~\ref{lem:RelShellRight}.
Then an ordering $\sigma_{1},\dots,\sigma_{m}$ is a relative shelling
if and only if the following holds:
\begin{equation}
\forall j\,\forall\tau\in\Delta_{j-1}\,\exists\tau_{*}\in\Delta_{j-1}\text{ such that }\tau\cap\sigma_{j}\subseteq\tau_{*}\cap\sigma_{j}=\sigma_{j}\setminus\left\{ x\right\} ,\text{ some }x\in\sigma_{j}.\label{eq:RelShellCond}
\end{equation}

We notice that we may restrict $\tau_{*}$ to be a facet of $\Delta_{j-1}$
in (\ref{eq:RelShellCond}) without loss of generality, indeed, $\tau_{*}$
may be selected to be either one of $\sigma_{1},\dots,\sigma_{j-1}$
or a facet of $\Gamma$. Since facets of $\Gamma$ are not necessarily
facets of $\Delta$, however, we cannot necessarily select $\tau_{*}$
to be a facet of $\Delta$.
\begin{rem}
The condition of (\ref{eq:RelShellCond}) and Lemma~\ref{lem:RelShellRight}
are exactly the same as those for extending a partial shelling to
a full shelling of a simplicial complex. Thus, we may think of relative
shellability as giving a setting where we may pretend we have already
shelled $\Gamma$ (whether that is possible or not), and must continue
the shelling order for the remaining facets (subject to the usual
facet-attachment conditions).
\end{rem}

Our main use of relative simplicial complexes and relative shellings
is as a tool to build (absolutely) shellable simplicial complexes.
The key observation for this endeavor is the following lemma.
\begin{lem}[Gluing Lemma]
 \label{lem:GluingRelShellings}Suppose that $\left(\Delta_{a},\Gamma_{a}\right)$
and $\left(\Delta_{b},(\Delta_{a}\cap\Delta_{b})\cup\Gamma_{b}\right)$
are shellable relative simplicial complexes, where $\Delta_{a}$ and
$\Delta_{b}$ are both contained in a common ambient supercomplex.
If $\Delta_{a}\cap\Gamma_{b}\subseteq\Gamma_{a}$ and $(\Delta_{a}\cap\Delta_{b})\cup\Gamma_{b}\neq\emptyset$,
then the relative complex $\Psi=\left(\Delta_{a}\cup\Delta_{b},\Gamma_{a}\cup\Gamma_{b}\right)$
is shellable.
\end{lem}

\begin{proof}
We begin with the shelling order of $\left(\Delta_{a},\Gamma_{a}\right)$.
Since $\left(\Gamma_{a}\cup\Gamma_{b}\right)\cap\Delta_{a}=\Gamma_{a}$,
this is also a partial shelling order of $\Psi$. We follow with the
shelling order of $\left(\Delta_{b},\left(\Delta_{a}\cap\Delta_{b}\right)\cup\Gamma_{b}\right)$,
using Lemma~\ref{lem:RelShellRight}.
\end{proof}
\begin{example}
\label{exa:2triangleRelTree}Let $\Delta_{a}$ be the triangle with
vertices $1,2,3$, and $\Delta_{b}$ be the triangle with vertices
$2,3,4$. It is easy to see that $\Delta_{a}$ is shellable relative
to any tree having at least $1$ edge. Using Lemma~\ref{lem:GluingRelShellings},
we see that the two-triangle complex $\Delta_{a}\cup\Delta_{b}$ is
shellable relative to any tree $\Gamma$ having at least $1$ edge.
For without loss of generality, we may assume that $\Delta_{a}$ contains
at least $1$ edge of $\Gamma$, and that $\Delta_{b}$ has at most
$1$ edge of $\Gamma$ not in $\Delta_{a}$. Now $\Gamma_{b}=(\Delta_{a}\cap\Delta_{b})\cup(\Gamma\cap\Delta_{b})$
is a tree in $\Delta_{b}$. Taking $\Gamma_{a}=\Delta_{a}\cap\Gamma$,
we see that the conditions of the lemma are met.
\end{example}

It is worthwhile to observe explicitly that, since absolute shellings
form a special case of relative shellings, Lemma~\ref{lem:GluingRelShellings}
may be used to combine an absolute shelling and a relative shelling
into an absolute shelling.

In order to apply Lemma~\ref{lem:GluingRelShellings}, we will need
various relative shellings. It will often be more convenient for us
to find absolute shellings, and apply the following lemma.
\begin{lem}
\label{lem:RelShellFromShell}Let $\Delta$ be a pure $d$-dimensional
simplicial complex with shelling order $\sigma_{1},\dots,\sigma_{m}$.
Furthermore, let $\Gamma$ be a pure $(d-1)$-dimensional subcomplex
of $\Delta$ such that for any facet $\tau$ of $\Gamma$ and any
$\sigma_{j}$ in the shelling order, one of the following holds:
\begin{enumerate}
\item \label{enu:RSgrowface}$\tau\cap\sigma_{j}\subsetneq\tau'\cap\sigma_{j}$
for some facet $\tau'$ of $\Gamma$, or
\item \label{enuRSshell}$\tau\cap\sigma_{j}\subseteq\sigma_{i}\cap\sigma_{j}$
for some $i<j$, or
\item \label{enuRScontain}$\tau\subseteq\sigma_{j}$.
\end{enumerate}
Then $\sigma_{1},\dots,\sigma_{m}$ is a shelling of the relative
complex $\left(\Delta,\Gamma\right)$.
\end{lem}

\begin{proof}
Let $\tau$ be a facet of the complex $\Delta_{j-1}$ generated by
$\sigma_{1},\dots,\sigma_{j-1}$ together with $\Gamma$. It suffices
to show that there is another facet $\tau'$ of $\Delta_{j-1}$ so
that $\tau'\cap\sigma_{j}$ is a $(d-1)$-face. If $\tau\cap\sigma_{j}\subseteq\sigma_{i}$
for $i<j$, then this follows from the definition of shelling. Otherwise,
$\tau$ is a facet of $\Gamma$ for which (\ref{enuRSshell}) does
not hold, and the desired is immediate by (\ref{enu:RSgrowface})
or (\ref{enuRScontain}).
\end{proof}
\begin{rem}
In simple language, the condition of Lemma~\ref{lem:RelShellFromShell}
requires that every maximal intersection of $\sigma_{j}$ with $\Gamma$
is either a facet of $\Gamma$ (so $(d-1)$-dimensional), or else
contained in some earlier facet in the shelling.
\end{rem}

\begin{example}
A \emph{shedding vertex} $v$ of $\Delta$ has the property that if
$v$ is in a face $\sigma$, then there is some other vertex $w$
so that $(\sigma\setminus v)\cup w$ is a face. It is well-known that
if $v$ is a shedding vertex such that $\Delta\setminus v$ and $\link_{\Delta}v$
are both shellable, then also $\Delta$ is shellable \cite{Bjorner/Wachs:1997,Wachs:1999b}.
This fact for pure complexes follows also from Lemma~\ref{lem:GluingRelShellings},
where we take $\Delta_{a}=\Delta\setminus v$, $\Delta_{b}=v*\link_{\Delta}v$,
and $\Gamma_{a}=\Gamma_{b}=\emptyset$. Now $\left(v*\link_{\Delta}v,\link_{\Delta}v\right)$
is shellable by Lemma~\ref{lem:RelShellFromShell} and the shelling
order on $v*\link_{\Delta}v$, where we take $\tau'=\sigma_{j}\setminus v$
whenever (\ref{enuRSshell}) fails. We recover that $\Delta$ is shellable
relative to $\Gamma_{a}\cup\Gamma_{b}=\emptyset$, i.e., $\Delta$
is (absolutely) shellable.

A similar argument applies to a shedding face $\gamma$, setting $\Delta_{a}=\Delta\setminus\gamma$
and $\Delta_{b}=\gamma*\link_{\Delta}\gamma$, so that $\Delta_{a}\cap\Delta_{b}=(\gamma*\link_{\Delta}\gamma)\setminus\gamma$.
As we will not use this result, we omit the details.
\end{example}

We use the following observation on gluing simplicial complexes freely
and without explicit reference, but state it here for clarity. The
proof is immediate from definitions.
\begin{lem}
\label{lem:GlueSimplicial}Let $\Delta_{a}$ and $\Delta_{b}$ be
simplicial complexes on disjoint vertex sets $V(\Delta_{a})$ and
$V(\Delta_{b})$. Let $\left\{ v_{1},\dots,v_{k}\right\} \subseteq V(\Delta_{a})$
and $\left\{ w_{1},\dots,w_{k}\right\} \subseteq V(\Delta_{b})$,
and let $\Gamma_{a}$ and $\Gamma_{b}$ respectively be the subcomplexes
induced by these vertex subsets. By identifying each $v_{i}$ with
$w_{i}$, we form a simplicial complex $\Sigma$, where topologically
$\Sigma$ is formed by gluing $\Delta_{a}$ and $\Delta_{b}$ along
$\Gamma_{a}\cap\Gamma_{b}$.
\end{lem}

We refer to \cite{Bjorner:1995} for additional background and definitions
on simplicial combinatorics.

\section{\label{sec:Turbines}Turbines and blades}

In this section, we construct the complexes $\Tau^{(n)}$ of Theorem~\ref{thm:TurbinesExist}.
We call the complex $\Tau^{(n)}$ the \emph{$n$-turbine, }for reasons
that will be apparent from Figure~\ref{fig:nTurbine}.

\subsection{\label{subsec:1turbine}The $1$-turbine}

First, Hachimori's example (pictured in Figure~\ref{fig:Hachimori})
will be $\Tau^{(1)}$. Hachimori verified his example to be shellable
in his thesis \cite{Hachimori:2000}, and it is clear from inspection
that it has a single free face. It is immediate from Lemma~\ref{lem:RelShellFromShell}
that if $\Upsilon$ is generated by any edge of the initial facet
in a shelling, then $(\Tau^{(1)},\Upsilon)$ is relatively shellable.
It is straightforward to find a shelling that begins with a facet
intersecting with the free edge at a vertex; we have shown one such
in Figure~\ref{fig:Hachimori}, relative to the subcomplex $\Upsilon=\left\langle sw\right\rangle $.
This proves Theorem~\ref{thm:TurbinesExist} for the case $n=1$.
As previously mentioned, this was already substantively observed in
\cite{Goaoc/Patak/Patakova/Tancer/Wagner:2019}.
\begin{rem}
The underlying topological space of Hachimori's example is a main
ingredient of our constructions. Since Hachimori based his construction
on the dunce cap space, we propose the \emph{tricorne cap} as a name
for this space.
\end{rem}

\begin{figure}
\includegraphics[scale=0.7]{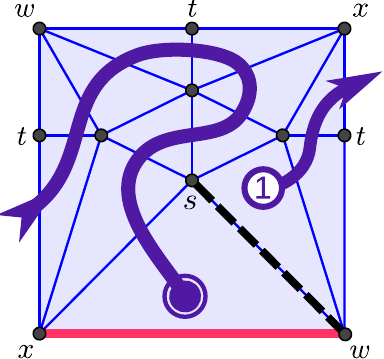}

\caption{\label{fig:Hachimori}The tricorne cap space, with a slight variation
of the triangulation of Hachimori. The arrows indicate a shelling
order that begins with the facet labeled $\textcircled{\scriptsize{1}}$
and ends with that labeled $\textcircled{\text{\ensuremath{\bullet}}}$.
The dashed edge is (a possible choice for) $\Upsilon$.}
\end{figure}

\subsection{\label{subsec:TurbineConstruction}Construction of $\protect\Tau^{(n)}$}

We will construct $\Tau^{(n)}$ for higher $n$ by gluing together
several copies of the tricorne cap space. We will need a triangulation
of $\Beta$ having $3$ adjacent free edges, shown in Figure~\ref{fig:BladeShellings}.
As our $n$-turbines will comprise $n$ copies of $\Beta$ glued around
a central ``shaft'', we call $\Beta$ the \emph{blade complex}.

\begin{figure}
\includegraphics[scale=0.7]{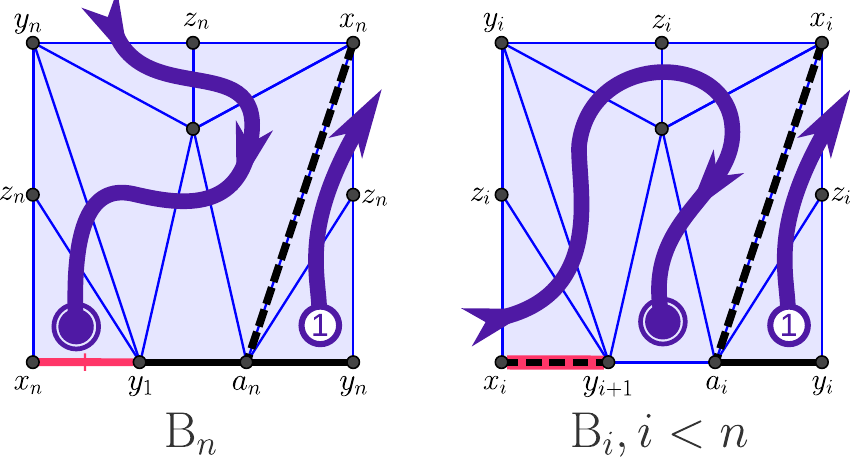}

\caption{\label{fig:BladeShellings}The blade complex $\protect\Beta$, pictured
with two shelling orders. Each shelling order begins with the facet
labeled $\textcircled{\scriptsize{1}}$ and ends with that labeled
$\textcircled{\text{\ensuremath{\bullet}}}$ in the respective diagram.
See also Corollary~\ref{cor:BladeRelShellings}.}
\end{figure}

Having constructed the blade, we now construct the $n$-turbine $\Tau^{(n)}$
for $n\geq3$. We begin with an $n$-cycle, having vertices $y_{1},y_{2},\dots,y_{n}$.
Subdivide each edge of the cycle by adding a vertex $a_{i}$ between
$y_{i}$ and $y_{i+1}$ (index considered mod $n$). Cone over the
subdivided $n$-gon to get a $2$-dimensional disc. Finally, for every
subdivided edge $y_{i},a_{i},y_{i+1}$, glue a copy $\Beta_{i}$ of
$\Beta$ along two adjacent free edges. A detailed schematic diagram
of the resulting construction can be found in Figure~\ref{fig:nTurbine}.

We also need to construct the tree subcomplex $\Upsilon$. Denote
by $w$ the apex vertex of the cone over the subdivided $n$-gon in
$\Tau^{(n)}$. For each blade $\Beta_{i}$ let $x_{i}$ be the vertex
of a free edge that is not glued to the central disc. The complex
$\Upsilon$ will be generated by all edges of the form $a_{i}w$ together
with those of the form $a_{i}x_{i}$; it is pictured with bold dark
edges in Figure~\ref{fig:nTurbine}.
\begin{rem}
We found the triangulation $\Beta$ of the tricorne cap space by first
subdividing the free edge in $\Tau^{(1)}$, and then applying cross-flips
and bistellar reductions in the sense of Pachner \cite{Pachner:1987}.
See also the systematic application of bistellar reductions developed
by Lutz in his thesis \cite{Lutz:1999}, and applied by Björner and
Lutz in \cite{Bjorner/Lutz:2000}. The authors find it interesting
that there is a triangulation of the tricorne cap with $3$ free edges
and only $6$ vertices, while they have been unable to find a triangulation
with single free edge and fewer than $7$ vertices.
\end{rem}

\begin{figure}
\includegraphics[scale=0.7]{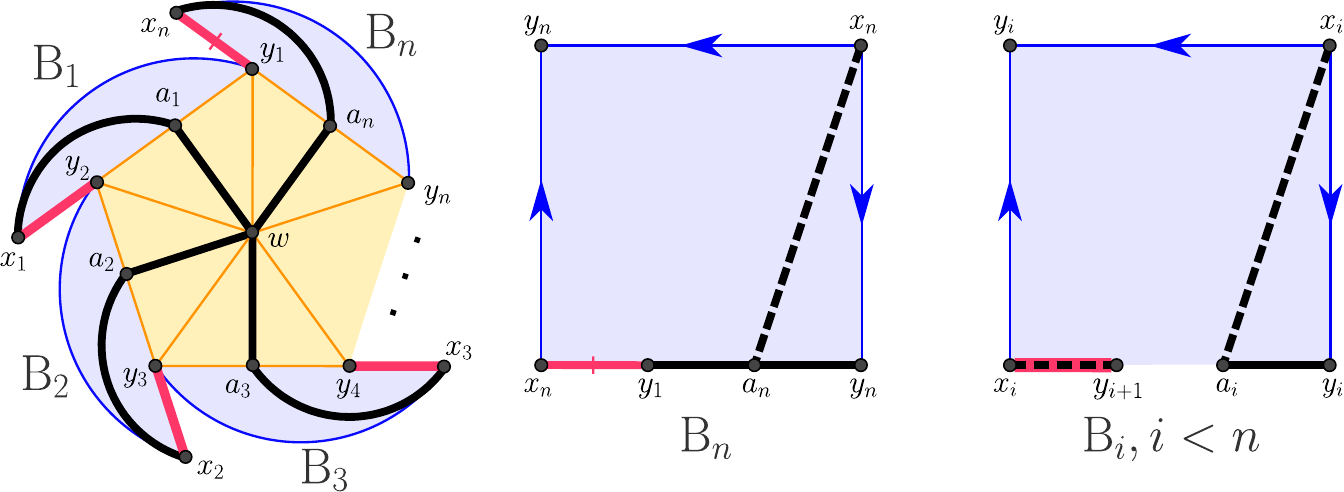}

\caption{\label{fig:nTurbine}The $n$-turbine space $\protect\Tau^{(n)}$
and its component blades $\protect\Beta_{i}$. For clarity, the blades
are shown schematically (with the details of the triangulation of
the blades omitted).}
\end{figure}

The case $n=2$ will require a slight variation of our main construction.
The complex obtained by gluing two copies of $\Beta$ to the cone
over a $4$-gon in the above manner is not simplicial, since we do
not wish to identify the edges $y_{1}y_{2}$ that are in each copy
of $\Beta$. To fix this, we subdivide the $y_{1}y_{2}$ edges in
both copies. All other details for $n=2$ proceed in exactly the same
way as for higher $n$. We picture $\Tau^{(2)}$ in Figure~\ref{fig:2Turbine},
where we take $\Upsilon$ to be the path formed by the bolded dark
edges.

\begin{figure}
\includegraphics[scale=0.7]{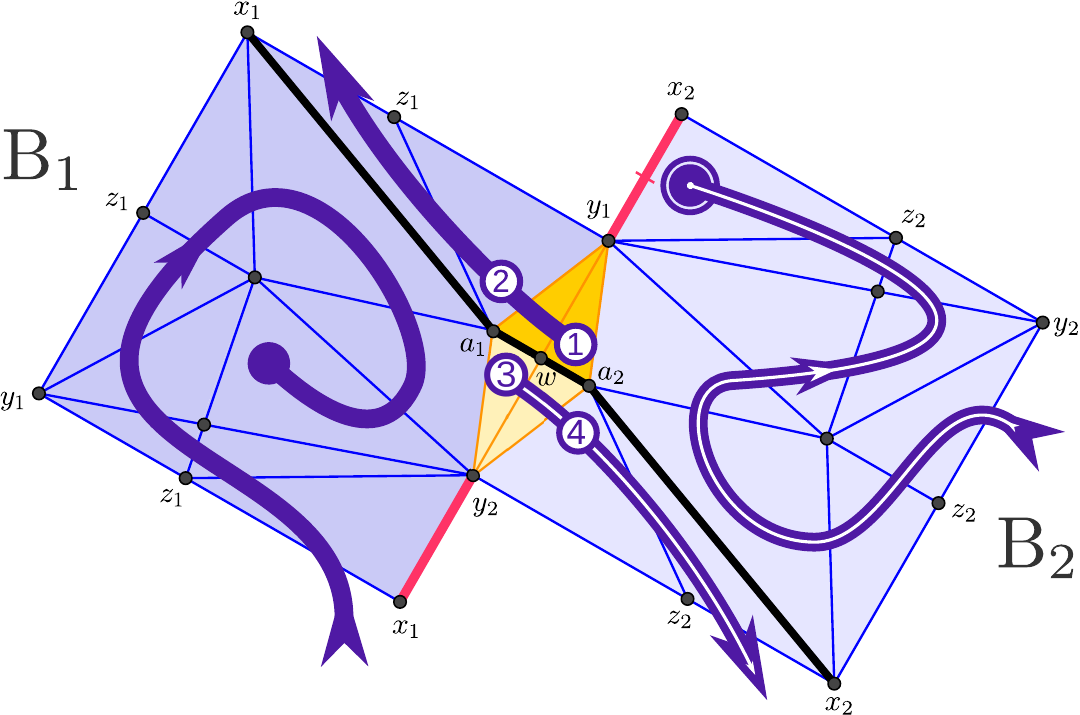}

\caption{\label{fig:2Turbine} The $2$-turbine space $\protect\Tau^{(2)}$.}
\end{figure}

\subsection{Proof of Theorem~\ref{thm:TurbinesExist} for $n=2$}

Although it is not difficult to verify that the order in Figure~\ref{fig:2Turbine}
that begins with facet $\textcircled{\scriptsize{1}}$ is a shelling,
we prefer to break the complex down using relative shellability. Our
strategy is to break $\Tau^{(2)}$ into several simpler subcomplexes,
relatively shell each of them, and use Lemma~\ref{lem:GluingRelShellings}
to glue the relative shellings together.

First, the two facets beginning with $\textcircled{\scriptsize{1}}$
clearly form a shelling, and the order is also a shelling relative
to $\Upsilon$. Second, the (modified) blade $\Beta_{1}$ is shellable
with the indicated order (beginning with $\textcircled{\scriptsize{2}}$,
indicated with the solid arrow). It follows from Lemma~\ref{lem:RelShellFromShell}
that $\Beta_{1}$ is shellable relative to the following subcomplexes:
$\left\langle a_{1}y_{1}\right\rangle $, $\left\langle a_{1}y_{1}\right\rangle \cup\Upsilon$,
and $\left\langle a_{1}y_{1},x_{1}y_{2}\right\rangle \cup\Upsilon$.
By Lemma~\ref{lem:GluingRelShellings}, the union of the initial
two facets and $\Beta_{1}$ are shellable, and also shellable relative
to $\Upsilon$ and $\left\langle a_{1}y_{1},x_{1}y_{2}\right\rangle \cup\Upsilon$.
In a similar way, the two facets beginning with $\textcircled{\scriptsize{3}}$
are shellable relative to $\left\langle a_{1}y_{2}\right\rangle $
or $\left\langle a_{1}y_{2}\right\rangle \cup\Upsilon$ . Thus, we
can use Lemma~\ref{lem:GluingRelShellings} to verify that the concatenation
of the first two facets, the shelling of $\Beta_{1}$, and that of
$\textcircled{\scriptsize{3}}$ and its successor is a shelling. Finally,
the pictured order of $\Beta_{2}$ beginning with $\textcircled{\scriptsize{4}}$
is a shelling (indicated with the hollow arrow). By Lemma~\ref{lem:RelShellFromShell},
this is also a shelling relative to $\left\langle a_{2}y_{2}\right\rangle $
or $\left\langle a_{2}y_{2}\right\rangle \cup\Upsilon$. Another application
of Lemma~\ref{lem:GluingRelShellings} gives that the pictured order
on $\Tau^{(2)}$ is a shelling, and also a shelling relative to $\Upsilon$
and $\Upsilon\cup\left\langle x_{1}y_{2}\right\rangle $. Since the
shelling of $\Tau^{(2)}$ has no homology facet, the complex is contractible.
The $n=2$ case of Theorem~\ref{thm:TurbinesExist} follows from
these (relative) shellings and symmetry of the complex.

\subsection{\label{subsec:ProofMainThmLargen}Proof of Theorem~\ref{thm:TurbinesExist}
for $n\protect\geq3$}

Our proof of Theorem~\ref{thm:TurbinesExist} for $n\geq3$ will
be entirely similar to the $n=2$ case. We first need two shellings
of $\Beta$, one for the final copy of $\Beta$ that is built up in
the shelling order, and one for all other copies of $\Beta$. By inspection,
we observe:
\begin{prop}
\label{prop:BladeShellings}The facet orders pictured in Figure~\ref{fig:BladeShellings}
are shellings of $B$.
\end{prop}

Lemma~\ref{lem:RelShellFromShell} lets us easily move to desired
relative shellings:
\begin{cor}
\label{cor:BladeRelShellings} In the following list of facets orders
and relative simplicial complexes, each order is a shelling of the
associated relative complexes.
\begin{enumerate}
\item The left-pictured facet order of Figure~\ref{fig:BladeShellings},
for the relative complexes $\left(\Beta,\left\langle a_{n}y_{n},a_{n}y_{1}\right\rangle \right)$
and $\left(\Beta,\left\langle a_{n}y_{n},a_{n}y_{1},a_{n}x_{n}\right\rangle \right)$.
\item The right-pictured facet order of Figure~\ref{fig:BladeShellings},
for the relative complexes $\left(\Beta,\left\langle a_{i}y_{i}\right\rangle \right)$,
$\left(\Beta,\left\langle a_{i}y_{i},a_{i}x_{i}\right\rangle \right)$,
and $\left(\Beta,\left\langle a_{i}y_{i},a_{i}x_{i},x_{i}y_{i+1}\right\rangle \right)$.
\end{enumerate}
\end{cor}

The subcomplexes of Corollary~\ref{cor:BladeRelShellings} are pictured
in bold and/or dashed edges in Figure~\ref{fig:BladeShellings}.

For each $i$, let $\alpha_{i}$ and $\beta_{i}$ be respectively
the facets $a_{i-1}y_{i}w$ and $a_{i}y_{i}w$ of $\Tau^{(n)}$, and
let $\Sigma_{i}$ be the complex spanned by $\alpha_{i}$ and $\beta_{i}$.
We begin the shelling with $\alpha_{1},\beta_{1}$. By Lemma~\ref{lem:GluingRelShellings},
we can continue with $\Beta_{1}$ in the ordering given by the first
relative shelling of Corollary~\ref{cor:BladeRelShellings}~(2).
Now by applying the simple argument from Example~\ref{exa:2triangleRelTree},
we can follow that with the shelling $\alpha_{2},\beta_{2}$ of $\Sigma_{2}$
(relative to $\left\langle a_{1}y_{2}\right\rangle $).

Continue inductively in this manner, alternately adding $\Beta_{i}$,
followed by $\Sigma_{i+1}$, using Lemma~\ref{lem:GluingRelShellings}
to glue the shellings at each step. After $n-1$ such steps, we have
all of $\Tau^{(n)}$ except for $\Beta_{n}$. Now as $\Beta_{n}$
intersects the central cone in the two edges $a_{n}y_{n}$ and $a_{n}y_{1}$,
we use the relative shelling of Corollary~\ref{cor:BladeRelShellings}~(1)
with Lemma~\ref{lem:GluingRelShellings} to complete the absolute
shelling of $\Tau^{(n)}$. Since this shelling has no homology facets,
the complex is contractible.

The relative statement follows with exactly the same proof, but using
the 2nd or 3rd relative complex in each application of Corollary~\ref{cor:BladeRelShellings}
on $\Beta_{i}$, and relative to $a_{i}w$ in each $\Sigma_{i}$.
Part (\ref{enu:TurbinesRelShell}) follows by this relative shelling,
together with rotational symmetry of $\Tau^{(n)}$.

\subsection{Additional remarks on turbines}

It is straightforward to count the faces of $\Beta$. The $f$-vector
and $h$-vector are 
\begin{align*}
f(\Beta)= & \left(1,6,14,9\right), & h(\Beta)= & \left(1,3,5,0\right).
\end{align*}
Using this calculation, together with direct counting for $\Tau^{(1)}$
and $\Tau^{(2)}$, we see that 
\begin{align*}
f(\Tau^{(1)})= & \left(1,7,19,13\right), & h(\Tau^{(1)})= & \left(1,4,8,0\right),\\
f(\Tau^{(2)})= & \left(1,13,38,26\right), & h(\Tau^{(2)})= & \left(1,10,15,0\right),\\
f(\Tau^{(n)})= & \left(1,5n+1,16n,11n\right), & h(\Tau^{(n)})= & \left(1,5n-2,6n+1,0\right),\text{ for }n\geq3.
\end{align*}

We have several remarks about variations of our construction. First,
if we were only interested in a shelling of $\Tau^{(n)}$, and not
also in the more delicate relative shellability property of Theorem~\ref{thm:TurbinesExist}~(\ref{enu:TurbinesRelShell}),
then we could somewhat simplify the construction. Indeed, we could
simplify the blade construction to have only $2$ free edges, and
glue a single free edge of each blade to the each facet of a cone
over an $n$-gon. A similar argument to that in Section~\ref{subsec:ProofMainThmLargen}
gives shellability. Relative shellability, on the other hand, seems
to require each blade to have two facets adjacent to the central $2n$-gon:
an ``in'' face and an ``out'' face.

We also remark that higher-dimensional analogues of our construction
are possible, at least in some special cases. Adiprasito, Benedetti
and Lutz \cite[Section 2]{Adiprasito/Benedetti/Lutz:2017} have generalized
Hachimori's construction to arbitrary dimension $d\geq2$. The $d$-dimensional
blade analogue will be formed from their example, by subdividing the
free face into $3$ free faces. There are some technical difficulties
in forming the central ``shaft'' portion of a $d$-dimensional analogue
of a turbine, as the cyclic symmetry of the $n$-gon does not cleanly
generalize to higher dimensions. Special cases are easy to construct,
such as an analogue of $\Tau^{(2)}$.

The restriction of Theorem~\ref{thm:ShellNPcomplete} to any dimension
higher than $2$ follows immediately by observing that coning preserves
shellability, so we do not need complexes with such properties for
$\NPtime$-completeness. As we at present have no other application
for such complexes, we do not here pursue further higher-dimensional
analogues of Theorem~\ref{thm:TurbinesExist}.

\section{\label{sec:NPcomplete}$\protect\NPtime$-completeness of $\protect\SHELL$}

In this section, we assume general familiarity with the theory of
$\NPtime$-completeness and polynomial reductions, on the level of
\cite{Papadimitriou:1994}.

The decision problem $\SHELL$ asks, given a list of facets of an
abstract simplicial complex $\Delta$, whether there is a shelling
of $\Delta$. Given an ordering of the facets, checking the condition
(\ref{eq:ShellCond}) can certainly be done in polynomial time, so
$\SHELL$ is in $\NPtime$. It is well-known that the restriction
of the $\TSAT$ problem where every literal occurs at most twice is
$\NPtime$-complete \cite[Proposition 9.3]{Papadimitriou:1994}, and
we will give a polynomial reduction from this restricted $\TSAT$
to $\SHELL$. As our reduction will involve only $2$-dimensional
facets, this will prove Theorem~\ref{thm:ShellNPcomplete}. Our reduction
will be simpler in several aspects than that of \cite{Goaoc/Patak/Patakova/Tancer/Wagner:2019},
and our proofs will be phrased directly in terms of shellings (rather
than in terms of collapsings).

We begin with an overview of our reduction. It is usual to divide
$\NPtime$-hardness proofs into building blocks called \emph{gadgets}.
We will have a choice gadget, corresponding to a variable, which will
consist of a triangulated sphere with a $\Tau^{(1)}$ glued along
its free face to a portion of the equator. The choice gadgets are
glued together by identifying an edge in the $\Tau^{(1)}$'s, so that
the triangulated spheres intersect at a single vertex. We will also
have a constraint gadget, corresponding to a clause, which will be
a $\Tau^{(2)}$ or $\Tau^{(3)}$ (according to the number of literals
in the clause). These constraint gadgets are glued to the choice gadgets
of the corresponding variables by gluing the central vertex of the
tree $\Upsilon$ to the common vertex of the triangulated spheres
in the choice gadgets, by wrapping each branch of $\Upsilon$ around
a choice gadget's equator, and by gluing the free face to the upper
or lower hemisphere (depending on whether the literal in question
is negated). Precise details are in Section~\ref{subsec:Reduction}.

An assignment of variables will correspond with a selection of upper/lower
hemispheres, one from each choice gadget. We will show that such an
assignment is satisfying if and only there is a shelling that has
homology facets exactly in the selected hemispheres.

\subsection{Gadgets}

Our choice gadget will consist of three parts. The first part will
be a $\Tau^{(1)}$. The second part will be a $2$-dimensional disc
$D$ having a boundary vertex $x$ that is incident to at least $2$
interior vertices $y$ and $y'$. Such a $D$ may be obtained by subdividing
a triangle with vertices $x,w,a$ to get a new interior vertex $y$,
then subdividing the edge $ay$ to get a new interior vertex $y'$.
See Figure~\ref{fig:ChoiceGadget}.

The third part will be an isomorphic copy of $D$, which we label
$\neg D$, and in which we label the interior vertices $\neg y$ and
$\neg y'$.

We glue the discs $D$ and $\neg D$ along their boundaries, and glue
the free edge of the $\Tau^{(1)}$ to the edge $wx$. Here $x$ is
as above, and $w$ is a vertex that is adjacent to $x$ in the common
boundary of $D$ and $\neg D$. The \emph{gluing edge} of the choice
gadget will be the edge $sw$ from Figure~\ref{fig:Hachimori} in
its $\Tau^{(1)}$ complex. The discs $D$ and $\neg D$ we call the
(\emph{positive} and \emph{negative}) \emph{literal hemispheres}.

\begin{figure}
\includegraphics[scale=0.7]{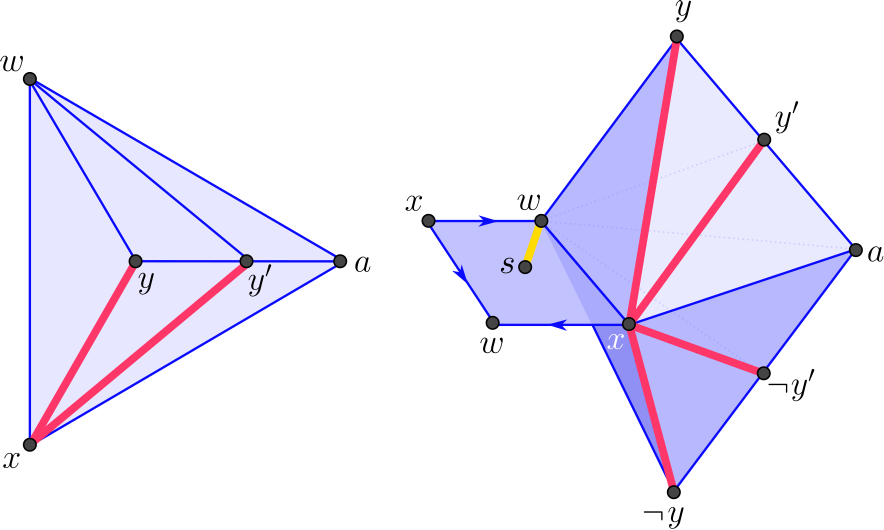}

\caption{\label{fig:ChoiceGadget}The literal hemisphere (left), and the choice
gadget (shown at right in 3D, with the $\protect\Tau^{(1)}$ complex
represented schematically). The choice gadgets are glued together
along the bold edge $ws$. The remaining bold edges are used to attach
free edges from constraint gadgets.}
\end{figure}

In the variation of $\TSAT$ where each literal occurs at most twice,
we must consider clauses with either two or three literals. For clauses
with two literals, our constraint gadget will be a $\Tau^{(2)}$;
for those with three literals, it will be a $\Tau^{(3)}$. We will
need the tree subcomplex $\Upsilon\cup\left\langle \mathcal{F}\right\rangle $
from Theorem~\ref{thm:TurbinesExist}~(\ref{enu:TurbinesTree}).
We call this subcomplex the \emph{gluing tree} of the constraint gadget.
As in Section~\ref{subsec:TurbineConstruction}, we may take the
gluing tree to consist of two or three branches $w,a_{i},x_{i},\tilde{y}_{i}$,
where each $x_{i}\tilde{y}_{i}$ is a free edge in the turbine. (Here
$\tilde{y}_{i}$ corresponds to $y_{i+1}\mod n$ in Figures~\ref{fig:nTurbine}
and \ref{fig:2Turbine}; $a_{i}$ and $x_{i}$ are as in the figures.)
\begin{rem}
\label{rem:ImprovementsOverPrior}Our choice gadgets consist of three
parts, fitting together in a simple way. For comparison, the choice
gadgets of \cite{Goaoc/Patak/Patakova/Tancer/Wagner:2019} have six
parts, and there are are some subtleties in how these parts are attached.
We also completely avoid the use of their consistency (`conjunction')
gadget.
\end{rem}

\subsection{\label{subsec:Reduction}Reduction}

The details of the reduction are now straightforward.

We are given a list of clauses, each with $2$ or $3$ variables.
Without loss of generality, each variable appears at most once in
each clause.

For each variable appearing in the list, we take a choice gadget.
We glue all of these gadgets together by identifying their gluing
edges to a single edge $sw$. Thus, the edge $sw$ is shared by all
of the choice gadgets, and the vertex $w$ is in every literal hemisphere.

Now for each clause, we attach a $\Tau^{(2)}$ or $\Tau^{(3)}$ constraint
gadget, according to whether the clause has $2$ or $3$ literals.
Thus, each literal $\ell$ of the clause corresponds to a branch of
the gluing tree. The literal also corresponds to a literal hemisphere
$D$ or $\neg D$ in a choice gadget. Let $y^{*}$ be the vertex $y$,
$y'$, $\neg y$, or $\neg y'$ of this literal hemisphere, where
we decorate $y$ with $\neg$ if $\ell$ is negated and with a prime
if this is the second clause containing $\ell$. (Recall that each
literal occurs in at most two clauses.) Now we glue the corresponding
branch $w,a_{i},x_{i},\tilde{y}_{i}$ of the gluing tree to the vertices
and edges $w,a,x,y^{*}$ of the literal hemisphere. That is, we glue
$w,a_{i},x_{i},\tilde{y}_{i}$ to the choice gadget along a path that
begins by wrapping around a portion of the equator (common to both
literal hemispheres), and whose final vertex is in the interior of
a literal hemisphere.

We remark that, since a given variable appears in a given clause at
most once, the vertices of the gluing tree attach to distinct vertices
in the union of choice gadgets.

Although we have described this reduction in terms of gluing, the
process admits a clear translation into facets, sets, and abstract
simplicial complexes. It is straightforwardly implemented in polynomial
time.

We denote by $\Delta$ the complex obtained by the polynomial reduction.
It is easy to calculate the homotopy type of $\Delta$:
\begin{lem}
\label{lem:HomTypeRedComplex} The complex $\Delta$ obtained by the
polynomial reduction is homotopy equivalent to a bouquet of $2$-spheres,
where the $2$-spheres are in bijective correspondence with the variables
in the $\TSAT$ instance.
\end{lem}

\begin{proof}
The union $\Delta_{\mathrm{lit}}$ of the literal hemispheres is exactly
a bouquet of simplicial spheres. The complex $\Delta$ can be obtained
from $\Delta_{\mathrm{lit}}$ by repeatedly attaching copies of the
contractible complexes $\Tau^{(1)},\Tau^{(2)},$ and $\Tau^{(3)}$
along a contractible (tree) subcomplex. The lemma now follows by well-known
results on gluing and homotopy type \cite[Lemma 10.3]{Bjorner:1995}.
\end{proof}

\subsection{\label{subsec:SatImpliesShell}Satisfiability implies shellability}

Suppose that our $\TSAT$ instance has a satisfying assignment, which
sets some literals to true and their negations to false. We will show
the complex $\Delta$ constructed in Section~\ref{subsec:Reduction}
is shellable, using repeated applications of Lemma~\ref{lem:GluingRelShellings}.
We will need the following (relative) shellings of the literal hemispheres:
\begin{lem}
\label{lem:LitHemisphereShell}The literal hemisphere disc $D$ is
shellable relative to the subcomplexes $\left\langle wx\right\rangle $,
$\partial D$, and $\partial D\cup\left\langle \mathcal{E}\right\rangle $,
where $\mathcal{E}$ is any subset of the edges $\left\{ xy,xy'\right\} $.
\end{lem}

\begin{proof}
All are immediate from Lemma~\ref{lem:RelShellFromShell}, using
any shelling order beginning with the facets $wxy,xyy'$.
\end{proof}
Given a satisfying assignment, we now build up the same complex $\Delta$
as in Section~\ref{subsec:Reduction}, but by gluing together subcomplexes
in a specific order that is compatible with Lemma~\ref{lem:GluingRelShellings}.
There are three main steps:\medskip{}

\emph{Step 1: `augmented' false literal hemispheres. }We begin with
the portion of the choice gadget consisting of $\Tau^{(1)}\cup D^{*}$,
where $D^{*}$ is either the positive or negative literal hemisphere.
By Lemmas~\ref{lem:GluingRelShellings} and \ref{lem:LitHemisphereShell}
and the discussion in Section~\ref{subsec:1turbine}, these \emph{augmented
literal hemispheres} are both shellable and also shellable relative
to $sw$. We take an augmented literal hemisphere for each false literal
in our assignment, and glue all copies together along the gluing edge
$sw$. By additional applications of Lemma~\ref{lem:GluingRelShellings},
this complex is shellable.

\emph{Step 2: constraint gadgets.} Next, for each clause in the $\TSAT$
instance, we attach a constraint gadget along the portion of its gluing
tree that is already present in the complex built so far. This portion
consists of $\Upsilon$, which attaches to the equators of false literal
hemispheres; together with the subset of the free edges $\mathcal{F}$
corresponding to the false literals in the clause, which attach to
the interiors of the corresponding false literal hemispheres. (See
the description in Section~\ref{subsec:Reduction}.) As every clause
contains at least one true literal, at least one edge of $\mathcal{F}$
is not glued. Thus, by Lemma~\ref{lem:GluingRelShellings} and Theorem~\ref{thm:TurbinesExist},
shellability is preserved under attaching each constraint gadget.

\emph{Step 3: true literal hemispheres.} Finally, we attach the true
literal hemispheres, yielding the full complex $\Delta$. By Lemma~\ref{lem:LitHemisphereShell}
and more applications of Lemma~\ref{lem:GluingRelShellings}, this
complex is shellable, as desired.

\medskip{}

We notice that the edge of $\mathcal{F}$ in each constraint gadget
that does not attach to an edge in the interior of a literal hemisphere
in Step 2 is crucial for this construction. Indeed, it is straightforward
to show that a turbine is not shellable relative to $\Upsilon\cup\left\langle \mathcal{F}\right\rangle $.
\begin{rem}
We observe that the complex built in this manner remains contractible
through Step 2, but that attaching each true literal hemisphere in
Step 3 creates a homology facet.
\end{rem}

\subsection{Shellability implies satisfiability}

Given a shelling of $\Delta$, and a subcomplex $\Gamma$ generated
by facets, we say that $\Gamma$ \emph{finishes shelling }at $\sigma$
if $\sigma$ is the last facet of the shelling that is contained in
$\Gamma$. A subcomplex may finish shelling either at a homology facet
of the shelling, or in a facet containing a free face in the partial
shelling. In the latter case, the free face $\tau$ obviously must
be free in $\Gamma$, and $\Gamma$ must finish shelling before any
other facet containing $\tau$ occurs in the shelling.

Now consider the complex $\Delta$ constructed by the polynomial reduction.
The facets of $\Delta$ consist of the disjoint union of the facets
of choice gadgets and of constraint gadgets. Moreover, the facets
of each choice gadget are the disjoint union of those of the $\Tau^{(1)}$,
those of the positive literal hemisphere, and those of the negative
literal hemisphere.

By Lemma~\ref{lem:HomTypeRedComplex}, we have a homology facet for
each variable. As removing the homology facets leaves a contractible
complex, there must be exactly one homology facet in each choice gadget,
contained in either the positive or negative literal hemisphere. For
each variable, we set the hemisphere containing the homology facet
to be the \emph{true} hemisphere, and the other one to be the \emph{false}
hemisphere. (Thus, a selection of homology facets corresponds, up
to equivalence, to a truth assignment of the variables.)

We now sort the subcomplexes formed by the $\Tau^{(1)}$'s, the literal
hemispheres, and the constraint gadgets according to when they finish
shelling. By results of Björner and Wachs \cite[Lemma 2.7]{Bjorner/Wachs:1996},
the homology facets of the complex may be taken without loss of generality
to come last in any shelling, thus the true literal hemispheres may
be assumed to finish shelling after all other considered subcomplexes.
Since each $\Tau^{(1)}$ has a unique free face, it must finish shelling
before its associated false literal hemisphere. Each false literal
hemisphere has only three free faces, with one attached to a $\Tau^{(1)}$,
and the other two attached to all of the constraint gadgets for clauses
containing the literal. Thus, a false literal hemisphere must finish
shelling before any constraint gadget for any clause containing the
literal. Since each constraint gadget has only two or three free faces,
each glued to an edge in a literal hemisphere, it must finish shelling
before  at least one of the literals in the corresponding clause.

Since each false literal finishes shelling before all the clauses
containing it, and each clause finishes shelling before at least one
literal contained in it, we must have that each clause contains a
true literal.
\begin{rem}
The authors of \cite{Goaoc/Patak/Patakova/Tancer/Wagner:2019} obtained
as a consequence of their main result that checking vertex-decomposability
is $\NPtime$-complete, and that $\SHELL$ is $\NPtime$-complete
even when restricted to order complexes of posets. We sketch how to
recover this result using our techniques, and our simpler construction.
Since the barycentric subdivision of a shellable simplicial complex
is a vertex-decomposable order complex \cite[Section 11]{Bjorner/Wachs:1997},
it suffices to show that if any subdivision of the complex $\Delta$
constructed by the polynomial reduction is shellable, then the underlying
formula is satisfiable. But as a free face in a subdivision of a simplicial
complex arises only by subdividing a free face in the original complex,
an argument entirely similar to the above shows that shellability
of a subdivision implies satisfiability.
\end{rem}

\subsection*{Data availability}

Data sharing not applicable to this article as no datasets were generated
or analysed during the current study.

\bibliographystyle{F}
\bibliography{2_Users_russw_Documents_Research_Master}

\end{document}